\documentclass[12pt]{article}
\usepackage{amsmath}
\usepackage{amsfonts}
\usepackage{amssymb}
\usepackage{amsthm}
\usepackage{stackrel}

\usepackage{color}

\usepackage{figsize}

\usepackage{graphics}

\newtheorem{thm}{Theorem}[section]

\newcommand{\R}{{\rm I}\kern-0.18em{\rm R}}
\newcommand{\1}{{\rm 1}\kern-0.25em{\rm I}}
\newcommand{\E}{{\rm I}\kern-0.18em{\rm E}}
\newcommand{\p}{{\rm I}\kern-0.18em{\rm P}}

\makeatletter

\def\@fnsymbol#1{\ensuremath{\ifcase#1\or a\or b\or c\or d\or \e\or f\or *\dagger 	\or \ddagger\ddagger \else\@ctrerr\fi}}
\title{Linear Statistics with Random Coefficients and Characterization of Hyperbolic Secant Distribution}
\author{Lev B. Klebanov\footnote{Department of Probability and Mathematical Statistics, Charles University, Prague, Czech Republic. e-mail: lev.klebanov@mff.cuni.cz}}
\date{}
\begin{document}
\maketitle

\begin{abstract}
There is given a characterization of hyperbolic secant distribution by the independence of linear forms with random coefficients. Also we provide a characterization by the identically distribution property.
\vspace{0.3cm}
\end{abstract}

\section{Introduction}\label{sec1}
\setcounter{equation}{0}

Many different characterizations of the normal distribution are known (see, for example \cite{KLR}). Some of them looks very natural and suitable for applications, especially for motivation of a distributions family choice. Let us mention four essential of them (from my point of view).  

\begin{enumerate}
\item[1N.]Normal distribution is only limit law for the sums of $n$ independent identically distributed (i.i.d.) random variables with finite variance. This is a consequence of well-known Central Limit Theorem. 
\item[2N.]Normal distribution is only one for which $X_1$ is identically distributed with $(X_1+X_2)/\sqrt{2}$, where $X_1, X_2$ are i.i.d. random variables. This statement is known as D. Polya theorem. Of course, there exist many generalizations of this theorem.
\item[3N.]Normal distribution is only one for which linear forms $X_1+X_2$ and $X_1-X_2$ are independent, where $X_1, X_2$ are i.i.d. random variables. The statement is known as S.N. Bernstein theorem. Its generalization is well-known Skitovich-Darmois theorem.
\item[4N.]Normal distribution is only one for which sample variance has a constant regression on sample mean. 
\end{enumerate}

One interesting (but not characteristic) property is that the density of normal distribution with zero mean has the same multiplicative type as its characteristic function up to normalizing constant.

Here I like to mention previous and new result connected to some similar properties of the hyperbolic secant distribution. These facts may be modified for characterizations of generalized hyperbolic secant distribution, however, I will not concentrate myself on any generalizations of such kind. They will be a subject of a separate publication.

\section{Some known characterizations of the hyperbolic secant distribution}\label{sec2}
\setcounter{equation}{0}

Let us mention some known results on characterization of the hyperbolic secant distribution. For simplicity, I consider the symmetric case only.

Firstly, for "the standard" case\footnote{In general case the density contains a scale parameter. We use term hyperbolic secant distribution for this general case.}, the density of hyperbolic secant distribution has the following form
\begin{equation}\label{eq1}
p(x)= \frac{1}{\pi}\cdot \frac{1}{\cosh(x)}.
\end{equation}
Corresponding characteristic function is
\begin{equation}\label{eq2}
f(t) = \frac{1}{\cosh(\pi t/2)}.
\end{equation}
Now we see, that the density of hyperbolic secant distribution has the same multiplicative type as its characteristic function up to normalizing constant. This is similar to corresponding property of normal distribution.

Let us mention other properties.
\begin{enumerate}
\item[1HS.]Hyperbolic secant distribution is only limit law for the sums of a special random number of i.i.d. random variables with finite variance. More precisely, let $X_1, \ldots ,X_n ,\ldots$ be a sequence of i.i.d. random variables with  finite variance and zero first moment. Suppose that $\nu_n,\; n\in \mathbb{N}$ is positive integer--valued random variate with probability generating function 
\begin{equation}\label{eq3}
\mathcal{P}_n(z)= \frac{1}{T_n(1/z)}.
\end{equation}
Then the sums $(1/\sqrt{n})\sum_{j=1}^{\nu_n}X_j$ converge in distribution to the hyperbolic secant distribution as $n \to \infty$. This fact was obtained in \cite{KKRT}.

\item[4HS.]There is a quadratic form of two (or more) i.i.d. random variables such that its regression on the sample mean is constant. This fact was firstly obtained in \cite{LL} (see also \cite{P} for more general facts).  
\end{enumerate}

Now we see that the properties 1HS and 4HS are very similar to that of normal law 1N and 4N. Below there are given analogs of 2N and 3N.

\section{Characterizations of hyperbolic secand distribution by the identical distribution and the independence properties}\label{sec3}
\setcounter{equation}{0}

Let us start with equally distribution property.

\begin{thm}\label{th1}
Let $X_1,X_2,X_3$ be i.i.d. random variables having a symmetric distribution. Suppose that $\varepsilon$ is Bernoulli random variable taking the values $1$ and $0$ with equal probabilities.
Suppose also that $\varepsilon$ is independent with $X_1,X_2,X_3$. The relation
\begin{equation}\label{eq4}
X_1\stackrel{d}{=} \frac{X_1+X_2}{2}+\varepsilon X_3 
\end{equation}
holds if and only if $X_1$ has hyperbolic secant distribution.
\end{thm}
\begin{proof}
The relation (\ref{eq4}) may be written in term of characteristic function $f(t)$ of random variable $X_1$ as
\begin{equation}\label{eq5}
f(t)=f^2(t/2)\frac{f(t)+1}{2}.
\end{equation} 
It is easy to see that if $f(t)$ satisfies (\ref{eq5}) then $f(at)$ is also a solution of (\ref{eq5}) for any positive $a$. Substituting (\ref{eq2}) into (\ref{eq5}) shows that the characteristic function of hyperbolic secant distribution is a solution of equation (\ref{eq5}).

Now we have to show that there are no other solutions of this equation. For that aim we use the method of intensively monotone operators (see \cite{KKM}). First of all, let us mention that if symmetric characteristic function $f$ is a solution of (\ref{eq5}) then it has no zeros on the real line. Really, if $f(t_o)=0$ then (\ref{eq5}) shows $f(t_o/2)=0$ and, therefore $f(t_o/2^k)=0$ for $k=1,2,\ldots$. in view of continuity of $f$ we must have $f(0)=0$ in contrary with $f(0)=1$.  

Define now the following operator
\[ (\mathbf{A}f)(t)=f^2(t/2)\frac{f(t)+1}{2}. \]
Let $\mathcal{E}$ be a set of non-negative functions, which are continuous in $[0,T]$ ($T>0$ is a fixed number). It is easy to see that the operator $\mathbf{A}$ is intensively monotone from $\mathcal{E}$ to the space of continuous functions on $[0,T]$. We use terminology from \cite{KKM}.

Consider the family $\{ \varphi(a t),\; a>0\}$, where $\varphi (t) = 1/\cosh(\pi t/2)$. According to Example 1.3.1 from \cite{KKM} this family is strongly $\mathcal{E}$-positive.

The statement of Theorem \ref{th1} follows now from \cite{KKM} Theorem 1.1.1.
\end{proof}

Let us note that the property given by Theorem \ref{th1} is similar to 2N, that is to D. Polya theorem.

Let us now provide a characterization of hyperbolic secant distribution by the property of independence of linear statistics with random coefficients.

\begin{thm}\label{th2}
Let $X_1,X_2,X_3$ be i.i.d. random variables having a symmetric distribution. Suppose that $\varepsilon$ is Bernoulli random variable taking the values $1$ and $0$ with equal probabilities.
Suppose also that $\varepsilon$ is independent with $X_1,X_2,X_3$. Consider two liner forms 
\begin{equation}\label{eq6}
L_1 = (X_1+X_2)/2+\varepsilon X_3; \quad L_2 = (X_1-X_2)/2+(1-\varepsilon) X_3.
\end{equation}
The forms (\ref{eq6}) are stochastic independent if and only if $X_1$ has hyperbolic secant distribution. 
\end{thm}
\begin{proof}
Calculate mutual characteristic function of the forms (\ref{eq6}). We have
\[ \E \exp\{isL_1+itL_2\}=\E \exp\{i\Bigl(\frac{s+t}{2}X_1+\frac{s-t}{2}X_2 +\bigl(\varepsilon s+ (1-\varepsilon)t\bigr)X_3\Bigr)\}=
\]
\begin{equation}\label{eq7}
=f((s+t)/2)f((s-t)/2)\frac{f(s)+f(t)}{2},
\end{equation}
where $f$ is characteristic function of $X_1$. The forms $L_1$ and $L_2$ are independent if and only if
\[ \E \exp\{isL_1+itL_2\}=\E \exp\{isL_1\}\E \exp\{itL_2\}, \]
which is equivalent to
\begin{equation}\label{eq8}
f((s+t)/2)f((s-t)/2)\frac{f(s)+f(t)}{2}=h_1(s)h_2(t).
\end{equation}
Because $X_1$ has a symmetric distribution we have
\begin{equation}\label{eq9}
h_1(s)=h_2(s)=f^2(s/2)\frac{f(s)+1}{2}.
\end{equation}
This expression can be obtained by direct calculations or by substituting $t=0$ (or $s=0$) into (\ref{eq8}). However, substituting $t=s$ into (\ref{eq8}) we find
\begin{equation}\label{eq10}
f^2(s)=\Bigl(f^2(s/2)\frac{f(s)+1}{2}\Bigr)^2.
\end{equation}
Similarly to the proof of Theorem \ref{th1} we see that $f$ has no zeros on real line. Therefore,
equation (\ref{eq10}) is equivalent to (\ref{eq2}). Theorem \ref{th1} shows now that the condition of independence of the forms (\ref{eq3}) implies hyperbolic secant distribution of $X_1$.

Direct calculations lead to the fact that (\ref{eq7}) holds for characteristic function of hyperbolic secant distribution.
\end{proof}

It is clear that Theorem \ref{th2} is similar to Theorem by S.N. Bernstein, mentioned above as 3N. 

\section*{\small{ACKNOWLEDGEMENT}}

The work was partially supported by Grant GA\v{C}R 19-04412S.

\end{document}